\documentclass[11pt,a4paper]{article}
\usepackage[utf8]{inputenc}
\usepackage[english]{babel}
\usepackage{amsmath,amssymb,amsthm}
\usepackage{geometry}

\newtheorem{theorem}{Theorem}[section]
\newtheorem{remark}[theorem]{Remark}
\newtheorem{proposition}[theorem]{Proposition}

\numberwithin{equation}{section}

\begin{document}

\title{\textbf{Calabi-Bernstein type theorems for entire solutions of certain nonlinear elliptic equations in the plane}}
\author{Michael Bildhauer \quad Martin Fuchs}
\date{}
\maketitle

\begin{abstract}
We provide some extensions of the classical Calabi-Bernstein theorem to the case of entire solutions of certain nonlinear elliptic equations in two variables. Using an appropriate version of the Calabi correspondence, which formulates a duality between minimal graphs in $\mathbb{R}^3$ and maximal spacelike graphs in $\mathbb{L}^3$, our results are a consequence of recent Bernstein type results.
\end{abstract}


\section{Introduction}

The classical Bernstein theorem for nonparametric minimal surfaces states that the affine functions are the only entire solutions $u : \mathbb{R}^n \rightarrow \mathbb{R}$ of the equation
\begin{equation}
\text{div} \left( \frac{\nabla u}{\sqrt{1 + |\nabla u|^2}} \right) = 0 \quad \text{on } \mathbb{R}^n
\end{equation}
as it was shown by Bernstein for the case $n=2$ in his famous papers [1] and [2] around a hundred years ago. Actually Bernstein's result extends to dimensions $n \le 7$, which geometrically means that for $n \le 7$ any entire minimal graph in Euclidean space $\mathbb{R}^{n+1}$ must be a hyperplane. This fundamental insight was established stepwise by Fleming [3], De Giorgi [4], Almgren [5] and Simons [6] in the sixties of the last century, and in 1969 Bombieri, De Giorgi and Giusti [7] constructed an example showing the failure of Bernstein's theorem for dimensions $n \ge 8$. We refer the reader being interested in the history of the Bernstein problem and its geometrical background to the monographs of Dierkes, Hildebrandt, Küster, Wohlrab [8], [9], of Nitsche [10] and of Osserman [11].

Passing from nonparametric minimal surfaces in Euclidean space $\mathbb{R}^{n+1}$ to maximal spacelike graphical hypersurfaces in the Lorentz-Minkowski space $\mathbb{L}^{n+1}$, equation (1.1) has to be replaced by
\begin{equation}
\text{div} \left( \frac{\nabla u}{\sqrt{1 - |\nabla u|^2}} \right) = 0 \quad \text{on } \mathbb{R}^n
\end{equation}
together with the requirement that
\begin{equation}
|\nabla u(x)| < 1 \quad \text{for all } x \in \mathbb{R}^n.
\end{equation}
The underlying geometric background leading to equation (1.2) and an appropriate explanation of the terminology can be found for instance in [12], [13], [14], [15], [16], [17], [18], [19] and [20]. Given equation (1.2) endowed with condition (1.3) one might ask, if a Bernstein type result is still true.

Surprisingly it turns out that for any dimension $n \ge 2$ entire solutions of (1.2), (1.3) must be affine functions. This fundamental result, known as Calabi-Bernstein theorem, was presented in 1968 and published 1970 in [15] by Calabi for the case $n \le 4$ and extended in 1976 by Cheng and Yau to dimensions $n > 4$ (see [20]).

The purpose of our note is to replace (1.2) by a class of slightly more general nonlinear elliptic equations (combined with adequate variants of condition (1.3)) and to prove Calabi-Bernstein type theorems at least for the twodimensional case. The limitation that $n = 2$ is required for the rest of this paper and enters for two reasons: first, we will make use of an adjusted version of the Calabi correspondence (see, e.g., [21], [15], [22], [23], [24] and compare for instance Lemma 3.2 in [24]), which allows us to transform our problem to the investigation of the Bernstein property of entire solutions $u$ of the Euler-Lagrange equations associated to variational integrals being of linear growth with respect to $|\nabla u|$. Unfortunately this duality argument seems to be limited to the case $n = 2$. Secondly, even if it would be possible to transform Calabi-Bernstein problems into Bernstein type problems in arbitrary dimensions, we could not benefit from this results, since up to now Bernstein type theorems are only available in two dimensions, we refer to [25] and [26], [27], [28].

After these preliminary remarks we are going to fix our assumptions: consider a function $a : [0, t_0) \rightarrow \mathbb{R}$ being of class $C^2$ on the interval [0, $t_0$), where $t_0$ denotes a number in $(0, \infty)$. We require
\begin{equation}
a'(0) = 0, \quad \lim_{t \rightarrow t_0} a'(t) = \infty, \quad a''(t) > 0 \quad \text{for all } t \ge 0,
\end{equation}
and note that $a$ is strictly convex. Moreover, $a'$ is a strictly increasing mapping from $[0, t_0)$ onto $[0, \infty)$.

In addition to (1.4) we assume
\begin{equation}
\liminf_{t \rightarrow t_0} \frac{a''(t)}{(a'(t))^\mu} > 0
\end{equation}
for some exponent $\mu > 1$. With these data we introduce the problem
\begin{equation}
\text{div} \left( \frac{a'(|\nabla u|)}{|\nabla u|} \nabla u \right) = 0 \quad \text{in } \mathbb{R}^2
\end{equation}
together with the gradient constraint
\begin{equation}
|\nabla u(x)| < t_0 \quad \text{for all } x \in \mathbb{R}^2.
\end{equation}
Note that (1.6) is just the Euler-Lagrange equation associated with the variational integral $\int a(|\nabla u|) \, dx$ and that we can recover the maximal surface case described in (1.2), (1.3) by letting $a(t) := -\sqrt{1 - t^2}$ for the choices $t_0 = 1$ and $\mu = 3$ in conditions (1.4) and (1.5). Our first result is an extension of the classical Calabi-Bernstein theorem to this more general situation.

\begin{theorem}
Let $n = 2$ and assume that conditions (1.4) and (1.5) are satisfied. Assume further that $u \in C^2(\mathbb{R}^2)$ is an entire solution of (1.6) together with (1.7). Then, if we require that
\begin{equation}
\mu \ge 3
\end{equation}
holds, $u$ must be an affine function.
\end{theorem}

\begin{remark}
We do not know, if Theorem 1.1 extends to dimensions $n \ge 3$ or if condition (1.8) can be weakened in the sense that the bound $\mu > 1$ is already sufficient for the Calabi-Bernstein property.
\end{remark}

Let us now replace (1.2) and (1.3) by a seemingly related problem: we consider a function $u \in C^2(\mathbb{R}^2)$ satisfying
\begin{equation}
\partial_1 \left( \frac{\partial_1 u}{\sqrt{1 - (\partial_1 u)^2}} \right) + \partial_2 \left( \frac{\partial_2 u}{\sqrt{1 - (\partial_2 u)^2}} \right) = 0 \quad \text{on } \mathbb{R}^2,
\end{equation}
\begin{equation}
|\partial_i u(x)| < 1 \quad \text{for } i = 1,2 \text{ and all } x \in \mathbb{R}^2
\end{equation}
and ask, if $u$ is an affine function. However, one easily checks that
\[
u(x_1, x_2) := \sqrt{1 + (x_1)^2} - \sqrt{1 + (x_2)^2}
\]
is a nonaffine entire solution of (1.9) and (1.10), which means that in the situation of (1.9) and (1.10) (``splitting case'') we have to impose some extra conditions guaranteeing that $\nabla u = \text{const}$. In this spirit it holds

\begin{theorem}
Consider functions $a_i : [0, t_i) \rightarrow \mathbb{R}$, $i = 1,2$, of class $C^2$ on $[0, t_i)$ with numbers $0 < t_i < \infty$ satisfying the conditions (1.4) and (1.5) with $t_0$ replaced by $t_i$ and exponents $\mu_i > 1$ again for $i = 1,2$. Let $u \in C^2(\mathbb{R}^2)$ solve
\begin{equation}
\partial_1 \left( \frac{a_1'(|\partial_1 u|)}{|\partial_1 u|} \partial_1 u \right) + \partial_2 \left( \frac{a_2'(|\partial_2 u|)}{|\partial_2 u|} \partial_2 u \right) = 0 \quad \text{on } \mathbb{R}^2,
\end{equation}
\begin{equation}
|\partial_i u(x)| < t_i, \quad i = 1,2, \text{ for all } x \in \mathbb{R}^2.
\end{equation}
Then, if for some constant $K > 0$ it holds
\begin{equation}
a_1'(|\partial_1 u|) \le K (a_2'(|\partial_2 u|) + 1) \quad \text{on } \mathbb{R}^2
\end{equation}
or
\begin{equation}
a_2'(|\partial_2 u|) \le K (a_1'(|\partial_1 u|) + 1) \quad \text{on } \mathbb{R}^2,
\end{equation}
$u$ is an affine function.
\end{theorem}

\begin{remark}
According to (1.13), (1.14) some kind of balancing property of the entire solution $u$ implies that $\nabla u = \text{const}$.
\end{remark}

The paper is organized as follows: in Section 2 we present some examples of functions $a : [0, t_0) \rightarrow \mathbb{R}$ with the properties (1.4) and (1.5), which according to Theorem 1.1 means that for these densities problem (1.6) in combination with (1.7) has the Calabi-Bernstein property at least for the choice $\mu \ge 3$ in condition (1.5). Section 3 is devoted to the proof of Theorem 1.1, where we make use of an appropriate variant of the Calabi correspondence combined with a Bernstein type result. The proof of Theorem 1.3, presented in Section 4, proceeds along similar lines. Moreover, we show how to construct counterexamples to the Calabi-Bernstein theorem in the splitting case. Finally, in Section 5 we deal with some extensions presenting a few additional results.

\section{Some examples of functions $a$ with (1.4) and (1.5)}

\noindent
i) Let $\alpha \in (0, 1)$ and consider
\[
a(t) := -\alpha t - (1 - t)^\alpha, \quad 0 \le t < t_0 := 1.
\]
Then (1.4), (1.5) hold with the choice $\mu := \frac{2-\alpha}{1-\alpha} > 1$, and we have (1.8) for exponents $\alpha \in [1/2, 1)$. \\

\noindent
ii) Again for $\alpha \in (0, 1)$ and $t \in [0, t_0)$, $t_0 := 1$, we define
\[
a(t) := - [1 - t^2]^\alpha
\]
with the same results as for example i). \\

\noindent
iii) Consider a function $h : [0, \infty) \rightarrow [0, \infty)$ of class $C^2$ and with the properties
\begin{equation}
\begin{cases}
h(0) = 0, \quad \lim_{t \rightarrow \infty} h(t) = \infty, \quad h'(0) = 0, \\
h''(t) > 0 \quad \text{for all } t \ge 0.
\end{cases}
\end{equation}
Clearly $h$ is a strictly increasing function from $[0, \infty)$ onto $[0, \infty)$. Letting $t_0 := h^{-1}(1)$ we introduce ($0 \le t < t_0$)\begin{equation}a(t) := - h^{-1}(1 - h(t)).\end{equation}It holds\begin{proposition}With $h$ satisfying (2.1) we let $a$ according to (2.2). Then (1.4) and (1.5) hold with the choice $\mu = 3$, which in particular means that we have the Calabi-Bernstein theorem for problem (1.6) equipped with (1.7).\end{proposition}The proof of Proposition 2.1 is an easy exercise, and clearly the last statement follows from Theorem 1.1.\begin{remark}Choosing $h(t) = t^2$ in (2.2) we recover problem (1.2), (1.3). The more exotic choice $h(t) := \exp(t) - t - 1$, $t \ge 0$, in formula (2.2) is of course also possible.\end{remark}\noindent
iv) In extension of example iii) we finally let\begin{equation}a(t) := - H(1 - h(t)), \quad 0 \le t < t_0,\end{equation}with $h$ according to (2.1) and as before $t_0 := h^{-1}(1)$. Then we have\begin{proposition}If the function $H : (0, 1] \rightarrow \mathbb{R}$ is of class $C^2$ with the properties\begin{equation}\begin{cases}H' > 0, \ H'' < 0 \text{ on } (0, 1], \ \lim_{t \rightarrow 0} H'(t) = \infty, \\lim\inf_{t \rightarrow 0} \frac{-H''(t)}{H'(t)^\mu} > 0 \text{ for some } \mu > 1,\end{cases}\end{equation}then the density $a$ defined in (2.3) satisfies (1.4) and (1.5).\end{proposition}Again, the proof of Proposition 2.3 is an easy exercise based on the assumptions (2.1) and (2.4). As a matter of fact the Calabi-Bernstein theorem follows from Theorem 1.1, if the choice $\mu \ge 3$ in (2.4) is possible.\begin{remark}If we let $H(t) := t^\alpha$, $0 < t \le 1$, with exponent $\alpha \in (0,1)$, then we obtain (2.4) for $\mu = \frac{2-\alpha}{1-\alpha}$, and if in addition we choose $h(t) = t^2$, then formula (2.3) yields example ii).\end{remark}\section{Proof of Theorem 1.1}Let us assume that condition (1.4) holds for the function $a : [0, t_0) \rightarrow \mathbb{R}$ and recall that in consequence $a'$ is a strictly increasing function, which maps the interval $[0, t_0)$ onto the set $[0, \infty)$. We define the function $g : [0, \infty) \rightarrow \mathbb{R}$ through \(g(\rho ):=\rho (a^{\prime })^{-1}(\rho )-a((a^{\prime })^{-1}(\rho ))\) and observe the formulas\begin{equation}\begin{cases}g'(\rho) = (a')^{-1}(\rho), \ g''(\rho) = 1 / a''((a')^{-1}(\rho)), \quad \rho \ge 0.\end{cases}\end{equation}From (3.1) we deduce the boundedness of $g'$ together with $g'(0) = 0$, moreover, $g''(\rho) > 0$ for all $\rho \ge 0$. In consequence, the function $g$ is of linear growth in the sense that\begin{equation}c_1 (\rho - 1) \le g(\rho) \le c_2 (\rho + 1), \quad \rho \ge 0,\end{equation}holds with positive constants $c_1$ and $c_2$. In fact, the convexity of $g$ shows $g(0) \ge g(\rho) - \rho g'(\rho)$, and the boundedness of $g'$ yields the second inequality in (3.2). For the first one we assume without loss of generality that $\rho \ge 1$ and get (again by convexity) $g(\rho) \ge g(1) + g'(1)(\rho - 1) = g'(1)\rho + g(1) - g'(1)$, which completes the proof of (3.2) by recalling that $g'(1) > 0$ on account of $g'(0) = 0$ and $g''(t) > 0$ for all $t \ge 0$. Altogether we see that condition (1.4) imposed on the density $a$ implies the hypotheses (1.7), (1.8) required for the function $g$ in Theorem 1.2 of [26]. Next we assume in addition to (1.4) that (1.5) is satisfied, thus it holds with some constant $c_3 > 0$ : \(a^{\prime \prime }(t)\ge c_{3}((a^{\prime }(t))^{\mu }+1),\quad 0\le t<t_{0}.\) Letting $t = (a')^{-1}(\rho)$, $\rho \ge 0$, and recalling (3.1) the last inequality implies \(g^{\prime \prime }(\rho )\le c_{4}(1+\rho )^{-\mu },\quad \rho \ge 0,\) with another constant $c_4 > 0$, hence condition (1.9) from [26] holds also true. Finally, we adjust the Calabi correspondence to the situation at hand: let $W : \mathbb{R}^2 \rightarrow \mathbb{R}^2$, \(W:=\frac{a^{\prime }(|\nabla u|)}{|\nabla u|}(\partial _{2}u,-\partial _{1}u),\) where $u \in C^2(\mathbb{R}^2)$ is an entire solution of (1.6), (1.7). From (1.6) we get $\partial_1 W_2 = \partial_2 W_1$, so that $W = \nabla w$ for a function $w \in C^2(\mathbb{R}^2)$. We have \(\frac{g^{\prime }(|\nabla w|)}{|\nabla w|}\nabla w=\frac{g^{\prime }(|W|)}{|W|}W=\frac{g^{\prime }(a^{\prime }(|\nabla u|))}{a^{\prime }(|\nabla u|)}W=(\partial _{2}u,-\partial _{1}u),\)which directly follows from (3.1) and the definition of $W$. This gives \(\text{div}\left(\frac{g^{\prime }(|\nabla w|)}{|\nabla w|}\nabla w\right)=0\quad \text{on\ }\mathbb{R}^{2}\) and under the additional hypothesis (1.8) we deduce from Theorem 1.2 in [26] that $w$ is an affine function. This implies the constancy of $W$ (recall $W = \nabla w$), and from the definition of $W$ the constancy of $\nabla u$ follows. This completes the proof of Theorem 1.1, but finally we would like to remark that the crucial Bernstein type result formulated in Theorem 1.2 of [26] is contained in Theorem 1.4 of the paper [25]. \qed\section{Proof of Theorem 1.3}Let the assumptions of Theorem 1.3 hold, which in particular means that we have the appropriate variants of the hypotheses (1.4), (1.5) for the densities $a_i : [0, t_i) \rightarrow \mathbb{R}$, $i = 1,2$. Following the arguments used in the previous section, we let \(g_{i}(\rho ):=\rho (a_{i}^{\prime })^{-1}(\rho )-a_{i}((a_{i}^{\prime })^{-1}(\rho )),\quad \rho \in \mathbb{R}_{0}^{+},\ i=1,2,\) and deduce the validity of (3.1), (3.2) for the functions $g_i$ together with the inequality\begin{equation}0 < g_i''(\rho) \le c(1 + \rho)^{-\mu_i}, \quad \rho \ge 0, \ i = 1,2.\end{equation}Consider $f : \mathbb{R}^2 \rightarrow \mathbb{R}$ given by \(f(p)=f(p_{1},p_{2}):=f_{1}(p_{1})+f_{2}(p_{2}),\)\(f_{1}(\rho ):=g_{2}(|\rho |),\quad f_{2}(\rho ):=g_{1}(|\rho |)\) for $p = (p_1, p_2) \in \mathbb{R}^2$ and $\rho \in \mathbb{R}$. By construction (in particular observe (4.1)) it is immediate that assumptions (3), (4) and (5) from [27] are true for the function $f$. Now consider an entire solution $u \in C^2(\mathbb{R}^2)$ of (1.11), (1.12) and let\begin{equation}W : \mathbb{R}^2 \rightarrow \mathbb{R}^2, \quad W := \left( \frac{a_2'(|\partial_2 u|)}{|\partial_2 u|} \partial_2 u, - \frac{a_1'(|\partial_1 u|)}{|\partial_1 u|} \partial_1 u \right),\end{equation}hence $\partial_2 W_1 = \partial_1 W_2$ and in conclusion $W = \nabla w$ for some $w \in C^2(\mathbb{R}^2)$. One easily checks that $\text{div}(\nabla f(\nabla w)) = 0$ holds on $\mathbb{R}^2$ and that (1.11), (1.12) imply that $w$ is of balanced type in the sense of Definition 1.1 from [27]. Therefore Theorem 1.4 in [27] shows that $w$ is an affine function and the constancy of $W$ immediately yields the same for $\nabla u$ completing the proof of Theorem 1.3. \qed \\\\Next we use the variant of the Calabi correspondence outlined above to construct non affine entire solutions of (1.11), (1.12) with densities $a_i : [0, t_i) \rightarrow \mathbb{R}$ as in Theorem 1.3. To this purpose we observe that $w(x_1, x_2) := x_1 x_2$ is an entire solution of the equation \(\text{div}(\nabla f(\nabla w))=\partial _{1}(f_{1}^{\prime }(\partial _{1}w))+\partial _{2}(f_{2}^{\prime }(\partial _{2}w))=0\) or equivalently \(\partial _{1}\left(\frac{g_{2}^{\prime }(|\partial _{1}w|)}{|\partial _{1}w|}\partial _{1}w\right)+\partial _{2}\left(\frac{g_{1}^{\prime }(|\partial _{2}w|)}{|\partial _{2}w|}\partial _{2}w\right)=0\) with $f_i$ and $g_i$ defined above. Suppose now that we have found a corresponding entire solution $u$ of (1.11), (1.12), which means that we have equation (4.2) for the field $W = \nabla w = (x_2, x_1)$ and in conclusion\begin{equation}x_2 = \frac{a_2'(|\partial_2 u|)}{|\partial_2 u|} \partial_2 u,\end{equation}\begin{equation}x_1 = - \frac{a_1'(|\partial_1 u|)}{|\partial_1 u|} \partial_1 u.\end{equation} From (4.3) we get $|x_2| = a_2'(|\partial_2 u|)$ or equivalently $|\partial_2 u| = (a_2')^{-1}(|x_2|)$. Inserting this into (4.3) we get \(\partial _{2}u=\frac{(a_{2}^{\prime })^{-1}(|x_{2}|)}{|x_{2}|}x_{2}.\) In the same manner (4.4) implies \(\partial _{1}u=-\frac{(a_{1}^{\prime })^{-1}(|x_{1}|)}{|x_{1}|}x_{1}.\) Conversely, one easily checks that \(u(x_{1},x_{2}):=\int _{0}^{x_{2}}\Phi _{2}(t)\,dt-\int _{0}^{x_{1}}\Phi _{1}(t)\,dt,\quad x\in \mathbb{R}^{2},\) \(\Phi _{i}(t):=\frac{(a_{i}^{\prime })^{-1}(|t|)}{|t|}t,\quad t\in \mathbb{R},\) is a nonaffine entire solution of (1.11), (1.12).\section{Some extensions}We start with a variant of Theorem 1.3, in which the hypothesis (1.13) (or (1.14)) is replaced by a sign condition.\begin{theorem}Consider functions  $a_i : [0, t_i) \rightarrow \mathbb{R}$ as in Theorem 1.3 and let $u \in C^2(\mathbb{R}^2)$ denote an entire solution of (1.11), (1.12). If $\partial_1 u$ or $\partial_2 u$ does not change sign (i.e. either $\partial_i u \ge 0$ or $\partial_i u \le 0$ for at least one index $i$), then $u$ is an affine function.\end{theorem}\begin{proof}We may assume that $\partial_1 u \ge 0$. Recalling $\nabla w = W$ with $W$ defined in (4.2) and remembering that $a_1' \ge 0$ (recall (1.4)) we find that $\partial_2 w \le 0$. Thus the function $w$ is an entire solution of the equation $\text{div}(\nabla f(\nabla w)) = 0$ (notation as in the proof of Theorem 1.3) with $\partial_2 w$ being bounded from one side and the results from [29] imply the constancy of $\nabla w$ and thereby of $\nabla u$.\end{proof}Finally we consider a rather general setting: let $\mathcal{O} \subset \mathbb{R}^2$ denote a bounded open and convex set and consider a function $f : \mathcal{O} \rightarrow \mathbb{R}$ of class $C^2$ such that\begin{equation}\nabla f : \mathcal{O} \rightarrow \mathbb{R}^2 \text{ is a diffeomorphism},\end{equation}\begin{equation}D^2 f(p)(q,q) > 0 \quad \text{for all } p \in \mathcal{O},  q \in \mathbb{R}^2,  q \neq 0.\end{equation}In place of (1.2), (1.3) or (1.6), (1.7) or (1.11), (1.12) we discuss the equation\begin{equation}\text{div}(\nabla f(\nabla u)) = 0 \quad \text{on } \mathbb{R}^2\end{equation}together with the gradient constraint\begin{equation}\nabla u(x) \in \mathcal{O} \quad \text{for } x \in \mathbb{R}^2.\end{equation}Suppose that $u \in C^2(\mathbb{R}^2)$ is a solution of (5.3), (5.4). As before we let \(W:=\left(\frac{\partial f}{\partial p_{2}}(\nabla u),-\frac{\partial f}{\partial p_{1}}(\nabla u)\right)\) and observe that according to (5.3) $\partial_2 W_1 = \partial_1 W_2$, hence $W = \nabla w$ with some $w \in C^2(\mathbb{R}^2)$. From (5.3) we also deduce the equation\begin{equation}\text{div}(\nabla g(\nabla w)) = 0 \quad \text{on } \mathbb{R}^2\end{equation}with function $g$ given by $g(p) := g(p_1, p_2) := \tilde{g}(-p_2, p_1)$, where we have set\begin{equation}\tilde{g}(p) := p \cdot (\nabla f)^{-1}(p) - f((\nabla f)^{-1}(p)), \quad p \in \mathbb{R}^2.\end{equation}It holds\begin{equation}\nabla \tilde{g}(p) = (\nabla f)^{-1}(p),\end{equation}\begin{equation}D^2 \tilde{g}(p) = \left[ D^2 f((\nabla f)^{-1}(p)) \right]^{-1}, \quad p \in \mathbb{R}^2,\end{equation}where in (5.8) $D^2 \tilde{g}(p)$ and $D^2 f(q)$, respectively, are seen as regular and symmetric $(2 \times 2)$-matrices, $[\dots]^{-1}$ denoting the inverse matrix (recall (5.2)). Note that equation (5.8) is a consequence of (5.7) and the definition of $g$. Concerning problem (5.3), (5.4) we now have\begin{theorem}Let (5.1), (5.2) hold and consider an entire solution $u \in C^2(\mathbb{R}^2)$ of (5.3), (5.4). Assume in addition that $f$ satisfies\begin{equation}\frac{\det D^2 f(q)}{(1 + |\nabla f(q)|)  |D^2 f(q)|} \ge \lambda, \quad q \in \mathcal{O},\end{equation}for some number $\lambda > 0$. Then, if we have\begin{equation}\left| \frac{\partial f}{\partial p_2}(\nabla u) \right| \le \Lambda \left( \left| \frac{\partial f}{\partial p_1}(\nabla u) \right|^m + 1 \right) \quad \text{on } \mathbb{R}^2\end{equation}or\begin{equation}\left| \frac{\partial f}{\partial p_1}(\nabla u) \right| \le \Lambda \left( \left| \frac{\partial f}{\partial p_2}(\nabla u) \right|^m + 1 \right) \quad \text{on } \mathbb{R}^2\end{equation}with numbers $\Lambda > 0$, $0 < m < 1$, $u$ must be affine.\end{theorem}\begin{proof}Recalling definition (5.6) of $\tilde{g}$ and observing that $|D^2 g(p_1, p_2)| = |D^2 \tilde{g}(-p_2, p_1)|$, we get from (5.8) \(|D^{2}g(p)|=\frac{|D^{2}f((\nabla f)^{-1}(-p_{2},p_{1}))|}{\det (D^{2}f((\nabla f)^{-1}(-p_{2},p_{1})))},\quad p\in \mathbb{R}^{2},\) which together with (5.9) ($q = (\nabla f)^{-1}(-p_2, p_1)$) yields \(|D^{2}g(p)|\le \lambda ^{-1}(1+|(-p_{2},p_{1})|)^{-1}=\lambda ^{-1}(1+|p|)^{-1}.\) This shows the validity of (1.5) and (1.6) from [28] for our density $g$. Moreover, (5.10) and (5.11) turn into (1.7) and (1.8) in [28] now for the function $w$, and [28], Theorem 1.1, gives the constancy of $\nabla w$. This implies the constancy of $\nabla f(\nabla u)$ and thereby of $\nabla u$.\end{proof}\ \begin{minipage}[t]{0.48\textwidth}

\textbf{Michael Bildhauer}\\
Saarland University\\
Department of Mathematics\\
P.O. Box 15 11 50\\
66041 Saarbrücken, Germany\\
\texttt{bibi@math.uni-sb.de}

\end{minipage}
\hfill
\begin{minipage}[t]{0.48\textwidth}

\textbf{Martin Fuchs}\\
Saarland University\\
Department of Mathematics\\
P.O. Box 15 11 50\\
66041 Saarbrücken, Germany\\
\texttt{fuchs@math.uni-sb.de}

\end{minipage}\end{document}